\documentclass[a4paper,10pt]{amsart}
\usepackage{amsmath,amsfonts,amssymb,amsthm, amsfonts}
\usepackage[latin1]{inputenc}
\usepackage[english]{babel}

\usepackage[cmtip,arrow]{xy}
\usepackage{pb-diagram,pb-xy}

\swapnumbers 
\theoremstyle{plain}
\newtheorem{thm}{Theorem}[section]
\newtheorem{prop}[thm]{Proposition}
\newtheorem{lemma}[thm]{Lemma}

\theoremstyle{remark}
\theoremstyle{definition}
\newtheorem{rem}[thm]{Remark}

\newtheorem{remdef}[thm]{Remark-Definition}

\newtheorem{defi}[thm]{Definition}
\newtheorem{defis}[thm]{Definitions}

\newtheorem{notas}[thm]{Notations}

\title[On Jordan-Chevalley decomposition]{Some remarks on \\
the Jordan-Chevalley decomposition}

\author{Alberto Dolcetti and Donato Pertici}
\address{Dipartimento di Matematica e Informatica ``Ulisse Dini''\\Viale Morgagni 67/a\\50134 Firenze, ITALIA}\email{alberto.dolcetti@unifi.it, \ donato.pertici@unifi.it}
\begin{document}
\parindent 0pt

\selectlanguage{english}

\begin{abstract}
In this note we mainly study the fine Jordan-Chevalley decomposition: a refinement of the classical Jordan-Chevalley decomposition of a matrix and we pay a particular attention to the field of the coefficients of the matrix. Moreover we obtain some further additive and multiplicative decompositions of a matrix under suitable conditions.
\end{abstract}

\maketitle

\tableofcontents

\renewcommand{\thefootnote}{\fnsymbol{footnote}}
\footnotetext{
This research was partially supported by MIUR-PRIN: ``Variet\`a reali e complesse: geo\-me\-tria, topologia e analisi armonica'' and by GNSAGA-INdAM.}
\renewcommand{\thefootnote}{\arabic{footnote}}
\setcounter{footnote}{0}


\medskip

{\scshape{Mathematics~Subject~Classification~(2010):} 15A21, 15A18, 12J10, 12J15.}

\section*{Introduction}

Aim of this note is to present some decomposition formulas for a square matrix $M$ starting from the classical \emph{Jordan-Chevalley decomposition} (or \emph{SN decomposition}).

Some of such formulas are well-known at least in ordinary real and complex cases. We set and, when possible, extend them in a coherent and self-contained context by paying attention to the properties of the field $\mathbb{K}$ of the coefficients of $M$.

In \S 1 we redraft the well-known additive \emph{Jordan-Chevalley decomposition} of $M$ as sum of its \emph{semisimple part} $S(M)$ and of its \emph{nilpotent part} $N(M)$ (Theorem \ref{SN-dec}). We construct the matrix $S(M)$ (and so $N(M) = M - S(M)$) as function of its \emph{Frobenius covariants} 
(Definitions \ref{def-partizione-I}) and of the eigenvalues of $M$. The Frobenius covariants of $S(M)$ are polynomial functions of $M$ uniquely determined by $M$ itself which we obtain by means of a suitable \emph{B\'ezout's identity} (Proposition \ref{ident-bezout-per-M}). We prove that an arbitrary matrix is semisimple if and only if it has a \emph{Frobenius decomposition}, i.e. it is linear combination of its Frobenius covariants with nonzero pairwise distinct coefficients (Proposition \ref{unicita-dec-Frob}). 
The matrices $S(M)$, $N(M)$ have coefficients in $\mathbb{K}^\dag$, the fixed field of $Aut(\mathbb{F} / \mathbb{K})$, where $\mathbb{F}$ is the spitting field of the minimal polynomial of $M$.

In \S 2 we decompose $S(M)$ in a unique way as sum of a finite number of suitable semisimple matrices $S_i(M)$'s with coefficients in $\mathbb{K}^\dag$, called \emph{unbreakable}. Each $S_i(M)$ is again a polynomial function of $M$ and it corresponds to  a distinct irreducible component over $\mathbb{K}$ of the minimal polynomial of $M$ or, equivalently, to a distinct orbit of the action of $Aut(\mathbb{F} / \mathbb{K})$ on the spectrum of $M$ (Definition \ref{unbreakable}, Proposition \ref{esistenza}, Remark \ref{S_i-unbreakable}).

In correspondence to each $S_i(M)$ we determine a suitable nilpotent matrix $N_i(M)$ with coefficients in $\mathbb{K}^\dag$. Such matrices, whose sum is $N(M)$, are polynomial functions of $M$, uniquely determined by suitable conditions (Notations \ref{S_i(M)N_i(M)}, Proposition \ref{N_i-uniche}).

Putting together these two decompositions, we get  \emph{the additive fine Jordan-Chevalley decomposition} of $M$ (Definition \ref{FineSN}), which seems to refine the Jordan-Chevalley decomposition.

Aim of \S 3 is to obtain the analogous of \emph{Schwerdtfeger's formula} and of \emph{Sylvester's formula}, which hold in real and complex cases and allow to express the image of a matrix under an analytic function $f$ by means of the derivatives of $f$, of the eigenvalues and of the Frobenius covariants of the matrix. This is fully obtained when $f$ is a polynomial, while, to guarantee the convergence of $f$ as a series, we assume that the field $\mathbb{K}$ is a \emph{complete valued field} with respect to a \emph{non-trivial absolute value} (Proposition \ref{Schw-Sylv}). The formula, we get, allows to identify easily the semisimple part and the nilpotent part of the image of the matrix. Finally we sketch how to get its fine components (Remark \ref{sketch-decomp-fine}).

Section \S 4 is devoted to \emph{real closed fields}, generalizing the real field. In particular we prove that if $M$ is a nonsingular matrix, then it has a unique \emph{complete multiplicative Jordan-Chevalley decomposition}, which expresses $M$ as a product of three pairwise commuting matrices which are polynomial expressions of $M$: a diagonalizable matrix over $\mathbb{K}$ with strictly positive eigenvalues, a semisimple matrix with eigenvalues of norm $1$ and a unipotent matrix (Proposition \ref{deltasigmau}). 
We conclude the section by proving, over the algebraic closure of such fields, the existence and the uniqueness of a coordinate-free version of the \emph{Singular Value Decomposition} of a matrix (Propositions \ref{SVD-esiste} and \ref{SVF-unica}).

\section{Jordan-Chevalley decomposition}

In this paper $\mathbb{K}$ denotes an arbitrary fixed field, $\overline{\mathbb{K}}$ its algebraic closure, $M_n(\mathbb{K})$ the $\mathbb{K}$-algebra of square matrices of order $n$ with coefficients in $\mathbb{K}$ and $I_n$ the identity matrix of order $n$.

\begin{defis}\label{def-partizione-I}
a) The \emph{spectrum}, $Sp(A)$, of a matrix $A  \in M_n(\mathbb{K})$ is the set of all eigenvalues of $A$ in $\overline{\mathbb{K}}$ and $Sp^*(A)$ is the set of nonzero elements of $Sp(A)$.

b) We recall that a matrix $A  \in M_n(\mathbb{K})$ is \emph{semisimple} if it is diagonalizable over $\overline{\mathbb{K}}$ or equivalently if it is diagonalizable over the splitting field of its minimal polynomial.

c) We say that a non-empty family of matrices $A_1, \cdots , A_p \in M_n(\overline{\mathbb{K}}) \setminus \{ 0 \} $ is a \emph{Frobenius system}, if $A_i A_j = \delta_{ij}A_i$ for every $i,j$ ($\delta_{ij}$ is the Kronecker symbol).

d) We call \emph{Frobenius decomposition of} $A \in M_n(\mathbb{K})$ any decomposition
\begin{center}
$A= \sum_{i=1}^p \lambda_i A_i$
\end{center}  
where $A_1, \cdots , A_p \in M_n(\overline{\mathbb{K}}) \setminus \{0\}$ form a Frobenius system and $\lambda_1, \cdots , \lambda_p \in \overline{\mathbb{K}} \setminus \{0\}$ are pairwise distinct. 
In this case, the matrices $A_1, \cdots , A_p$ are called \emph{Frobenius covariants} of $A$ (see for instance \cite{HoJ1991}, p. 403).
\end{defis}

\begin{lemma}\label{Propr-partizione}
Let $A \in M_n(\mathbb{K})$ and assume that it has a Frobenius decomposition $A= \sum_{i=1}^p \lambda_i A_i$. Then

a) $\overline{K}^n = Im A_1 \oplus Im A_2 \oplus \cdots \oplus Im A_p \oplus Ker A$;

b) the distinct nonzero eigenvalues of $A$ in $\overline{\mathbb{K}}$ are $\lambda_1, \cdots , \lambda_p$ (hence $A \ne 0$);

c) $Ker(A-\lambda_i I_n) = Im(A_i)$ for every $i$, $Ker(A) = Im(I_n -\sum_{i=1}^p A_i)$ and so $A$ is semisimple;

d) $0$ is an eigenvalue of $A$ if and only if $\sum_{i=1}^p rk(A_i) < n$ .
\end{lemma} 

\begin{proof} 
These facts are standard and their proofs can be found for instance in \cite{YTT2011} Ch.2.
\end{proof}

\begin{lemma}\label{due-partizioni}
Let $A_1, \cdots , A_l \in M_n(\mathbb{K}) \setminus \{0\}$, $l \ge 1$, and assume that, for every $i$,  $A_i= \sum_{j=1}^{p_i} \lambda_{ij} A_{ij}$ is a Frobenius decomposition of $A_i$ with $Sp^*(A_i) \cap Sp^*(A_{i'}) = \emptyset$ for every $i \ne i'$.

Then \, $\sum_{i=1}^l \sum_{j=1}^{p_i} \lambda_{ij} A_{ij}$ is a Frobenius decomposition of \, $\sum_{i=1}^l A_i$ if and only if $A_i A_{i'}=0$ as soon as $i \ne i'$.
\end{lemma}

\begin{proof}
One implication is trivial.

Assume now that $A_i A_{i'}=0$ for every $i \ne i'$ and remember that from the previous Lemma we have: 
$\overline{K}^n = Im A_{i1} \oplus  \cdots \oplus Im A_{i p_i} \oplus Ker A_i$ for every $i$.

Fix an index $i' \in \{1, \cdots , l \}$, an index $m \in \{1, \cdots , p_{i'} \}$ and a vector $w \in \overline{\mathbb{K}}^n$. 

For $i \ne i'$ we have: $0 = A_i A_{i'} (A_{i'm} w) = \sum_{h, k} \lambda_{ih} \lambda_{i'k} A_{ih} A_{i'k} (A_ {i'm} w) =$

$ \sum_{h}\lambda_{ih} \lambda_{i'm} A_{ih} A_{i'm} (A_{i'm} w) =  \sum_{h}\lambda_{ih} \lambda_{i'm} A_{ih} (A_{i'm} w)$, hence from the decompositions of $\overline{\mathbb{K}}^n$ above, we get: 
$\lambda_{ih} \lambda_{i'm} A_{ih} (A_{i'm} w) =0$ for every $h \in \{1, \cdots , p_i \}$. So $A_{ih} A_{i'm} w=0$ for every $w \in \overline{\mathbb{K}}^n$, being $\lambda_{ih} \lambda_{i'm} \ne 0$. Hence $\{A_{ij}\}$ is a Frobenius system and we can conclude because the $\lambda_{ij}$'s are nonzero and pairwise distinct.
\end{proof}

\begin{notas}[{\bf and remarks}]\label{pol-min}
a) Next $M$ will be always a fixed matrix in $M_n(\mathbb{K})$ with minimal polynomial 
\begin{center}
$m(X) =m_1(X)^{\mu_1} \cdots m_r(X)^{\mu_r}$
\end{center}
where $\mu_1, \cdots , \mu_r > 0$ and $m_1(X), \cdots , m_r(X)$ are mutually distinct irreducible polynomials in $\mathbb{K}[X]$ of degrees $d_1, \cdots , d_r$ respectively. 

Note that $0$ is an eigenvalue of $M$ if and only if one of the polynomials $m_i(X)$'s is $X$. From now on, in this case (after reordering) we assume that $m_r(X)=X$.

b) We denote by $\mathbb{F}$ the splitting field of $m(X)$ and by $\mathbb{K}^\dag$ the fixed field of the group $Aut(\mathbb{F} / \mathbb{K})$ (the group of automorphisms of $\mathbb{F}$ fixing each element of $\mathbb{K}$). 

Of course: $\mathbb{K} \subseteq \mathbb{K}^\dag \subseteq \mathbb{F} \subseteq \overline{\mathbb{K}}$ and inclusions are generally strict.

Moreover $\mathbb{K}= \mathbb{K}^\dag$ (i.e. $\mathbb{F} / \mathbb{K}$ is a \emph{Galois extension}, see for instance \cite{Lang2002} Ch. VI \S 1) if and only if each polynomial $m_i(X)$ is \emph{separable} over $\mathbb{K}$ (i.e. its roots in $\overline{\mathbb{K}}$ are all distinct). This is always true if $\mathbb{K}$ is \emph{perfect}, e.g. in case of characteristic $0$ (see for instance \cite{Kapl1972} p.26 and p.58).

Note that, by Jordan canonical form, $M$ is semisimple if and only if $\mu_1 = \cdots = \mu_r =1$ and $\mathbb{K} = \mathbb{K}^\dag$.

c) We denote by $\lambda_{i 1}, \cdots , \lambda_{i \rho_i}$ the $\rho_i$ distinct roots of $m_i(X)$ (in $\mathbb{F})$. 

We recall that, if $P(X) \in \mathbb{K}[X]$ is irreducible over $\mathbb{K}$, the subset of $\overline{\mathbb{K}}$ of all its distinct roots is said to be \emph{a conjugacy class over} $\mathbb{K}$. Therefore $\lambda_{i 1}, \cdots , \lambda_{i \rho_i}$ form a conjugacy class over $\mathbb{K}$.

By the assumption in (a), if $0$ is eigenvalue of $M$, then $\rho_r = 1$ and $\lambda_{r1} = 0$.
Hence $Sp(M)$ and $Sp^*(M)$ are both disjoint union of conjugacy classes over $\mathbb{K}$.

Note that $\lambda_{ij} \ne \lambda_{i'j'}$ as soon as $(i,j) \ne (i',j')$ and moreover we have $\rho_i \leq d_i$, generally without equality because the polynomials $m_i(X)$'s are not supposed to be separable.

Every element of $Aut(\mathbb{F} / \mathbb{K})$ acts as a permutation on the sets of roots of each polynomial $m_i$, so every $\varphi \in Aut(\mathbb{F} / \mathbb{K})$ induces a permutation $\sigma_i^{\varphi}$ on each set $\{ 1, \cdots , \rho_i \}$, $i = 1 , \cdots , r$ such that $\varphi(\lambda_{ij}) = \lambda_{i \, \sigma_i^{\varphi}(j)}$. 
 
Polynomials in $\mathbb{F}[X]$, which are invariant under the action of $Aut(\mathbb{F} / \mathbb{K})$ on their coefficients, are actually in $\mathbb{K}^\dag[X]$.

We can factorize $m(X) = \prod_{i=1}^r \prod_{j=1}^{\rho_i} (X-\lambda_{ij})^{\eta_i}$ for suitable integers $\eta_i \ge \mu_i$. 

The exponent $\eta_i$ is equal to $\mu_i$, if $m_i(X)$ is a separable polynomial; otherwise the characteristic of $\mathbb{K}$ is positive and $\dfrac{\eta_i}{\mu_i}$ is a power of it. In all cases the power $\eta_i$ of $(X-\lambda_{ij})$ is independent of $j$ (see for instance \cite{Lang2002} pp.284--285).

d) For every $i$, we denote $g_i(X) = \prod_{j=1}^{\rho_i} (X- \lambda_{ij})$ and 
$g(X) = \prod_{i=1}^r g_i(X)$.

The polynomials $g_i$ and $g$ are invariant under the action of the group $Aut(\mathbb{F} / \mathbb{K})$ (because the coefficients of each $g_i$ are elementary symmetric functions of the roots $\lambda_{i 1}, \cdots , \lambda_{i \rho_i}$), so they belong to $\mathbb{K}^\dag [X]$.

We remark that, if $\mathbb{L}'/ \mathbb{L}$ is any normal finite-dimensional extension, then the orbit of every $\alpha \in \mathbb{L}'$ under the action of $Aut(\mathbb{L}'/ \mathbb{L})$ coincides with the conjugacy class of $\alpha$ over $\mathbb{L}$ (this is a consequence for instance of \cite{Kapl1972} Thm.21 p.24).

Since $\mathbb{F}$ is a normal finite-dimensional extension of both $\mathbb{K}$ and $\mathbb{K}^\dag$ and $Aut(\mathbb{F}/ \mathbb{K}) = Aut(\mathbb{F}/ \mathbb{K}^\dag)$, the conjugacy classes of $\alpha$ over $\mathbb{K}$ and over $\mathbb{K}^\dag$ overlap.

Therefore, for every $i$, the set $\{\lambda_{i1}, \cdots , \lambda_{i\rho_i}\}$ is a conjugacy class over $\mathbb{K}^\dag$ too, hence $g_i$ is irreducible as element of $\mathbb{K}^\dag[X]$ and $\rho_i$ is the degree of each $\lambda_{ij}$ on $\mathbb{K}^\dag$.

e) Finally we pose $G_{ij}(X) = \dfrac{m(X)}{(X-\lambda_{ij})^{\eta_i}} = \prod_{h \ne i} m_h(X)^{\mu_h} \prod_{k \neq j} (X- \lambda_{ik})^{\eta_i}$ for every $i= 1 , \cdots , r$ and every $j = 1 , \cdots , \rho_i$ and we observe that
$G_{ij}(X)$ has coefficients in $\mathbb{K}(\lambda_{i1}, \cdots , \lambda_{i \rho_i}) \subseteq \mathbb{F}$.
\end{notas}

\begin{prop}\label{ident-bezout-per-M}
Assume the same notations as in \ref{pol-min}.

a) There is a unique set of polynomials $\{ B_{ij}(X) \ / \ i= 1 , \cdots , r, \ j = 1 , \cdots , \rho_i \}$ in $\mathbb{F}[X]$  such that
\begin{center}
$B_{ij}(X) \in \mathbb{K}(\lambda_{i1} \cdots , \lambda_{i \rho_i})[X], \  deg\, B_{ij}(X) < \eta_i \mbox{ for every } i,j$ \ and

$\sum_{i=1}^r \sum_{j=1}^{\rho_i} B_{ij}(X) G_{ij}(X) = 1$ \ \ \mbox{(B\'ezout's identity).}
\end{center}

b) For every $i= 1 , \cdots , r$ and $j = 1 , \cdots , \rho_i$, the polynomial 
\begin{center}
$C_{ij}(X) = B_{ij}(X) G_{ij}(X)$
\end{center}
is in $\mathbb{K}(\lambda_{i1} \cdots , \lambda_{i \rho_i})[X] \subseteq \mathbb{F}[X]$ and satisfies $deg \, C_{ij}(X) < deg \, m(X)$.

Moreover the family of matrices  $C_{ij}(M) \in M_n(\mathbb{F})$ ($i= 1 , \cdots , r, \ j = 1 , \cdots , \rho_i$) is a Frobenius system with 
\begin{center}
$
\sum_{i=1}^r \sum_{j=1}^{\rho_i} C_{ij}(M) = I_n
$.
\end{center}
\end{prop}

\begin{proof}
Since $G_{ij}(X)$ and $(X-\lambda_{ij})^{\eta_i}$ are polynomials with coefficients in 

$\mathbb{K}(\lambda_{i1}, \cdots , \lambda_{i \rho_i})$ and have greatest common divisor equal to $1$, there exist, uniquely determined, 
$B_{ij}(X), L_{ij}(X) \in \mathbb{K}(\lambda_{i1} \cdots , \lambda_{i \rho_i})[X]$ such that 

$B_{ij}(X) G_{ij}(X) + L_{ij}(X) (X-\lambda_{ij})^{\eta_i} = 1$, 
$deg(B_{ij}) < \eta_i$ and $deg(L_{ij}) < deg(G_{ij}) = deg(m) - \eta_i$, for every $i,j$. 

So $(X-\lambda_{ij})^{\eta_i}$ divides $B_{ij}(X) G_{ij}(X) -1$. On the other hand $(X-\lambda_{ij})^{\eta_i}$ divides $B_{hk}(X) G_{hk}(X)$ as soon as $(h,k) \ne (i,j)$ and so  $\sum_{h,k} B_{hk}(X) G_{hk}(X) - 1$ is divided by every $(X-\lambda_{ij})^{\eta_i}$ and hence by $m(X)$. 

But $deg(\sum_{h,k} B_{hk}(X) G_{hk}(X) - 1) < deg(m(X))$ and so $\sum_{h,k} B_{hk}(X) G_{hk}(X) = 1$.

For the uniqueness of the polynomials $B_{ij}(X)$'s, assume that certain polynomials $A_{ij}(X)$'s satisfy the same properties of the $B_{ij}(X)$'s. 
Hence $\sum_{h,k} A_{hk}(X) G_{hk}(X) = A_{ij}(X) G_{ij}(X) + \sum_{(h,k) \ne (i,j)} A_{hk}(X) G_{hk}(X) = 1$ and
$deg(A_{ij}) < \eta_i$. 

Since $(X-\lambda_{ij})^{\eta_i}$ divides $\sum_{(h,k) \ne (i,j)} A_{hk}(X) G_{hk}(X)$, we can write 

$A_{ij}(X) G_{ij}(X) + L_{ij}'(X) (X-\lambda_{ij})^{\eta_i} =1$ with $deg(A_{ij}) < \eta_i$ and 

$deg(L_{ij}') < deg(G_{ij}) = deg(m) - \eta_i$ and we conclude (a) by the uniqueness above recalled.

For (b) we remark that $\sum_{i,j} C_{ij} (M) =I_n$ is a direct consequence of (a). 

If $(i,j) \ne (h,k)$, then $C_{ij}(M) C_{hk}(M) = B_{ij}(M) B_{hk}(M) G_{ij}(M) G_{hk}(M) = 0$, because $G_{ij}(X) G_{hk}(X)$ is a multiple of the minimal polynomial of $M$. 

Finally $C_{ij} (M) = C_{ij} (M) I_n = C_{ij} (M) \left[\, \sum_{h,k} C_{hk} (M)\right] = C_{ij}(M)^2$ and the assertion is proved.
\end{proof}

\begin{thm}[additive Jordan-Chevalley decomposition]\label{SN-dec} 
Let $M \in M_n(\mathbb{K})$, $\lambda_{i j}$ be the distinct eigenvalues of $M$ as in Notations \ref{pol-min} and the matrices $C_{ij}(M)$'s as in \ref{ident-bezout-per-M}.
Then the matrices
$$S(M)= \sum_{i=1}^r  \sum_{j=1}^{\rho_i} \lambda_{ij}C_{ij}(M) \mbox{ \ and \ } N(M)= M-S(M)$$
 are polynomial expressions of $M$ and have coefficients in $\mathbb{K}^\dag$, 
$S(M)$ is semisimple, $N(M)$ is nilpotent, $S(M) N(M)= N(M) S(M)$ and of course
$M = S(M) + N(M)$.

Moreover if $M=S+N$ is any decomposition with $S \in M_n(\overline{\mathbb{K}})$ semisimple, 

$N \in M_n(\overline{\mathbb{K}})$ nilpotent and $S N = N S$, then $S=S(M)$ and $N=N(S)$. 
\end{thm}

\begin{proof}
Note that if $\varphi \in Aut(\mathbb{F} / \mathbb{K})$ and $\sigma_i^{\varphi}$ is the permutation induced by $\varphi$ on $\{ 1, \cdots , \rho_i \}$, $i = 1 , \cdots , r$, then $\varphi(B_{ij}(X)) = B_{i \, \sigma_i^{\varphi}(j)}(X)$ (we still denote by $\varphi$ its natural extension to $\mathbb{F}[X]$).

Indeed from the equality $1= \sum_{i=1}^r \sum_{j=1}^{\rho_i} B_{ij}(X) G_{ij}(X)$ we get: 

$
1= \varphi(1) = \varphi(\sum_{i=1}^r \sum_{j=1}^{\rho_i} B_{ij}(X) G_{ij}(X))= \sum_{i=1}^r \sum_{j=1}^{\rho_i} \varphi(B_{ij}(X)) \varphi(G_{ij}(X))
$.

Now  $\varphi(G_{ij}(X)) = G_{i \, \sigma_i^{\varphi}(j)}(X)$, because $\varphi$ acts as a permutation, hence 

$1=  \sum_{i=1}^r \sum_{j=1}^{\rho_i} \varphi(B_{ij}(X)) G_{i \, \sigma_i^{\varphi}(j)}(X)$ and $\varphi(B_{ij}(X)) = B_{i \, \sigma_i^{\varphi}(j)}(X)$ by uniqueness of $B_{ij}$'s in \ref{ident-bezout-per-M} (a).

In particular every $\varphi \in Aut(\mathbb{F} / \mathbb{K})$ satisfies: $\varphi(C_{ij}(X)) = C_{i\, \sigma_i^{\varphi}(j)}(X)$  for every $i,j$, because $\varphi$ preserves the product.

Now let us consider the polynomial $S(X) =  \sum_{i=1}^r  \sum_{j=1}^{\rho_i} \lambda_{ij}C_{ij}(X)$. 

For every $\varphi \in Aut(\mathbb{F} / \mathbb{K})$, we get $\varphi(S(X)) = \sum_{i=1}^r \sum_{j=1}^{\rho_i} \lambda_{i \, \sigma_i^{\varphi}(j)} C_{i \, \sigma_i^{\varphi}(j)}(X) = S(X)$, so $S(X) \in \mathbb{K}^\dag [X]$ being fixed by each $\varphi \in Aut(\mathbb{F} / \mathbb{K})$. 

This implies that $N(X)= X - S(X) \in  \mathbb{K}^\dag [X]$ and that the matrices $S(M)$ and $N(M)$ (in $M_n(\mathbb{K}^\dag)$) commute, because they are polynomial expressions of $M$.

$S(M)$ is diagonalizable on $\mathbb{F}$ by \ref{Propr-partizione} and \ref{ident-bezout-per-M} (b).

To prove that $N(M)$ is nilpotent, we remark that, by the properties of the matrices $C_{ij}(M)$'s, we have:
$[M-S(M)]^h = [\sum_{i,j} (M-\lambda_{ij}I_n)C_{ij}(M)]^h = \\
\sum_{i,j} (M-\lambda_{ij}I_n)^h C_{ij}(M) = (\sum_{i,j} (M-\lambda_{ij}I_n)^{h- \eta_i} B_{ij}(M)) m(M) =0$ as soon as $h \ge max\{\eta_1 , \cdots , \eta_r\}$.

Finally if $S(M) + N(M) = S + N$, then $S(M) - S = N - N(M)$. The condition $SN=NS$ implies that the matrices $S$ and $N$ commute with $M$ and so also with $S(M)$ and $N(M)$. Hence $S$ and $S(M)$ have  a common basis of eigenvectors in $\overline{\mathbb{K}}^n$, so $S(M) - S$ is semisimple. Since $N-N(M)$ is nilpotent, we conclude that $S(M) - S = N - N(M) = 0$.
\end{proof}

\begin{defi}
The matrices $S(M)$ and $N(M)$ of the previous Theorem are said to be the \emph{semisimple part} and the \emph{nilpotent part} of $M$ respectively and the decomposition $M= S(M) + N(M)$ is said to be the (\emph{additive}) \emph{Jordan-Chevalley decomposition} (or (\emph{additive}) \emph{SN decomposition}) of $M$.
\end{defi}

\begin{rem}\label{r'}
We denote by $r'$ the integer $r'=r$ when $M$ is nonsingular and $r'=r-1$ otherwise.
Remembering \ref{pol-min}, if $S(M) \ne 0$, we can write the semisimple part of $M$ as 
\begin{center}
$S(M)= \sum_{i=1}^{r'}  \sum_{j=1}^{\rho_i} \lambda_{ij}C_{ij}(M)$ , with $\lambda_{ij} \ne 0$ for every $i,j$.
\end{center}
This decomposition is a Frobenius decomposition of $S(M)$ and the matrices $C_{ij}(M)$ with $i = 1, \cdots , r'$ and $j=1 , \cdots , \rho_i$, are Frobenius covariants of $S(M)$.
\end{rem}

\begin{prop}\label{unicita-dec-Frob}
A matrix $A \in M_n(\mathbb{K}) \setminus \{0\}$ has a Frobenius decomposition if and only if it is semisimple.

If this is the case, the Frobenius decomposition and, so, the Frobenius covariants are uniquely determined.
\end{prop}

\begin{proof}
Indeed the first part follows directly from \ref{r'} and from \ref{Propr-partizione}. 

For the uniqueness, if $A = \sum_{i = 1}^p \lambda_i A_i$ is a Frobenius decomposition of $A$, then from \ref{Propr-partizione} the coefficients $\lambda_i$'s are necessarily the non-zero distinct eigenvalues of $A$ and from the properties of the $A_i$'s, arguing as in \cite{GaXu2002} Thm. 2.2, we get the matricial system
$$
\sum_{h=1}^p \lambda_h^m A_h = A^m, \  1 \le m \le p,
$$
whose associated matrix has non-zero determinant, because it is equal to the Vandermonde determinant of $\lambda_1 , \cdots , \lambda_p$ multiplied by $\lambda_1  \cdots  \lambda_p$. Hence the $A_i$'s are uniquely determined.
\end{proof}

\begin{remdef}\label{JC-moltiplicativa} 
a) Note that the matrix $M$ is not nilpotent if and only if $S(M) \ne 0$. In this case we refer to  the Frobenius covariants of $S(M)$ also as \emph{Frobenius covariants of} $M$.

b) The matrices $M$ and $S(M)$ above have the same distinct eigenvalues, so if $M$ is nonsingular, then $S(M)$ is nonsingular too. 

In this case we get easily the equality $S(M)^{-1} = \sum_{i=1}^r \sum_{j= 1}^{\rho_i} \lambda_{ij}^{-1} C_{ij}(M)$; this gives the Frobenius decomposition of the matrix $S(M)^{-1}$ which results polynomial in $M$. 

Moreover from the additive decomposition $M= S(M) + N(M)$, we get the  \emph{multiplicative Jordan-Chevalley decomposition} (or \emph{multiplicative SN decomposition}) $M=S(M) \, U(M)$ with $U(M)=(I_n+ S(M)^{-1}N(M)) \in M_n(\mathbb{K}^\dag)$ unipotent and $S(M), \, U(M)$ polynomials in $M$ (and therefore commuting).

Finally  $S(M), \, N(M)$ are the unique matrices with coefficients in $\overline{\mathbb{K}}$ such that $M=S(M) \, U(M)$, $S(M)$ semisimple, $U(M)$ unipotent and $S(M) \, U(M) = U(M) \, S(M)$ (see for instance \cite{Humph1975} Lemma B p.96).

By the way we note that in general $S(M)^h = \sum_{i=1}^r \sum_{j= 1}^{\rho_i} \lambda_{ij}^h C_{ij}(M)$ for every $h \in \mathbb{N}$ and that, if $M$ is nonsingular, then the same formula holds for every $h \in \mathbb{Z}$.
\end{remdef}

\section{Fine Jordan-Chevalley decomposition}

\begin{rem}
Let $S$ be a semisimple square matrix with coefficients in a generic field $\mathbb{L}$ and $\mathbb{F}$ be the splitting field of its minimal polynomial. Since $S$ is semisimple, its minimal polynomial has no multiple root in $\mathbb{F}$ (i.e. it is separable over $\mathbb{L}$). Hence $\mathbb{F}/\mathbb{L}$ is a Galois extension (see for instance \cite{Kapl1972} Part I, \S 5) and so the fixed field $\mathbb{L}^\dag$ of $Aut(\mathbb{F}/\mathbb{L})$ is exactly $\mathbb{L}$. 
\end{rem}

\begin{defi}\label{unbreakable}
We say that a semisimple matrix $S \in M_n(\mathbb{L})$ is \emph{unbreakable over} $\mathbb{L}$, if, whenever $S=A+B$ with $A, B \in M_n(\mathbb{L})$ semisimple matrices, $AB=BA=0$ and $Sp^*(A) \cap Sp^*(B) = \emptyset$, then $A=0$ or $B=0$.
\end{defi}

\begin{prop}\label{esistenza}
Let $S \in M_n(\mathbb{L})$ be a semisimple matrix.

1) $S$ is unbreakable over $\mathbb{L}$ if and only if $S=0$ (i.e. $Sp^*(S) = \emptyset$) or $Sp^*(S)$ consists in a single conjugacy class over $\mathbb{L}$.

2) If $Sp^*(S)$ consists of $l \ge 1$ conjugacy classes, then there exist $l$ unbreakable semisimple matrices, $S_1, \cdots, S_l \in M_n(\mathbb{L}) \setminus \{0\}$, such that 
$S = S_1 + \cdots + S_l$ and, for every $i \ne j$, $S_i S_j = 0$ and $Sp^*(S_i) \cap Sp^*(S_j) = \emptyset$.

3) If $S \ne 0$, its decomposition in unbreakable semisimple matrices stated in (2) is unique up to the order of the matrices $S_i$'s and we refer to it as \emph{the unbreakable semisimple decomposition of} $S$ and to matrices $S_i$'s as \emph{the unbreakable semisimple components of} $S$.
\end{prop}

\begin{proof}
Assume that $S=A+B$ with $A, B \in M_n(\mathbb{L})$ semisimple matrices, $AB=BA=0$ and $Sp^*(A) \cap Sp^*(B) = \emptyset$ with $A \ne 0$ and $B \ne 0$. Let $A = \sum_{i=1}^p \alpha_i A_i$ and $B = \sum_{j=1}^q \beta_j B_j$ be their Frobenius decompositions. Since $Sp^*(A) \cap Sp^*(B) = \emptyset$, by \ref{due-partizioni} we get that  $\sum_{i=1}^p \alpha_i A_i + \sum_{j=1}^q \beta_j B_j$ is the Frobenius decomposition of $S=A+B$. In particular, by \ref{Propr-partizione}, $\alpha_1$ and $\beta_1$ are both eigenvalues of $S$ with distinct conjugacy classes over $\mathbb{L}$: indeed the class of $\alpha_1$ in contained in $Sp^*(A)$ and the class of $\beta_1$ in contained in $Sp^*(B)$. This concludes part ``if'' of (1).

Now if $S$ is semisimple and nonzero, then we can consider its Frobenius decomposition: $\sum_{i=1}^l \sum_{j=1}^{\rho_i} \lambda_{ij} C_{ij}$, where $C_{ij}$'s are the Frobenius covariants of $S$ and $\{\lambda_{i1}, \cdots , \lambda_{i\rho_i}\}$, $1 \le i \le l$, are the $l$ conjugacy classes of nonzero eigenvalues of $S$.

For every $i$ we denote: $S_i = \sum_{j=1}^{\rho_i} \lambda_{ij} C_{ij}$.
These matrices have coefficients in $\mathbb{L}$.

For, since $\mathbb{F}/\mathbb{L}$ is a Galois extension, it suffices to check that every $S_i$ is $Aut(\mathbb{F}/\mathbb{L})$-invariant.

Arguing as in the proof of \ref{SN-dec}, if $\varphi \in Aut(\mathbb{F} / \mathbb{L})$ and $\sigma_i^{\varphi}$, $1 \leq i \le l$, is the permutation induced by $\varphi$ on $\{ 1, \cdots , \rho_i \}$, then $\varphi(\lambda_{ij}) = \lambda_{i \, \sigma_i^{\varphi}(j)}$ and $\varphi(C_{ij}) = C_{i \, \sigma_i^{\varphi}(j)}$. Hence $\varphi (S_i) = \sum_{j=1}^{\rho_i} \lambda_{i \sigma_i^{\varphi}(j)} C_{i \, \sigma_i^{\varphi}(j)} = S_i$ and therefore $S_i$ has coefficients in $\mathbb{L}$.

These matrices are unbreakable because of part ``if'' of (1), while the remaining properties follow easily from the properties of $C_{ij}$'s and from \ref{Propr-partizione}. This completes part (2).

Part ``only if'' of (1) is a direct consequence of (2).

For (3), let $S = S_1' + \cdots + S_{l'}'$ be another decomposition with every $S'_i \in M_n(\mathbb{L}) \setminus \{0\}$ semisimple and unbreakable over $\mathbb{L}$ and such that, for every $i \ne j$, $S'_i S'_j = 0$ and $Sp^*(S'_i) \cap Sp^*(S'_j) = \emptyset$.

By \ref{due-partizioni}, the sum of the Frobenius decompositions of the matrices $S'_j$'s is the Frobenius decomposition of $S$. But each matrix $S'_j$ is unbreakable and, by (1), it is uniquely determined by a conjugacy class over $\mathbb{L}$ of nonzero eigenvalues of $S$. Since the same fact holds for the matrices $S_i$'s, by uniqueness of the Frobenius decomposition of $S$, we get that $l=l'$ and, up to change order, $S_i = S'_i$ for every $i$.
\end{proof}

\begin{notas}[{\bf and remarks}]\label{S_i(M)N_i(M)}
a) Remembering the same notations as in \ref{pol-min}, in the remaining part of this section we assume that the fixed matrix $M \in M_n(\mathbb{K})$ is not nilpotent; this is equivalent both to $S(M) \ne 0$ and to $r' \ge 1$ (remember \ref{r'}).

b) By \ref{Propr-partizione}, the eigenspace in  $\overline{\mathbb{K}}^n$, relative to the eigenvalue $\lambda_{ij}$ of $S(M)$, is $Im \, C_{ij}(M)$. 
Hence $\overline{\mathbb{K}}^n = \oplus_{i=1}^r \oplus_{j=1}^{\rho_i} Im \, C_{ij}(M)$ as vector spaces over $\overline{\mathbb{K}}$ (remember that, by \ref{Propr-partizione} (c), if $\lambda_{r1}= 0$, then $Im \, C_{r1}(M) = Ker \, S(M)$).

c) If the polynomials $C_{ij}(X)$'s are as in \ref{ident-bezout-per-M}, we define the polynomials 
$$
S_i(X) = \sum_{j=1}^{\rho_i} \lambda_{ij} C_{ij}(X)
$$
for every $i=1, \cdots , r'$ and 
$$N_i(X) = \sum_{j=1}^{\rho_i} (X-\lambda_{ij}) C_{ij}(X)
$$
for every $i=1, \cdots , r$

and the related matrices $S_i(M)$ and $N_i(M)$. 

We note that $\sum_{j=1}^{\rho_i} \lambda_{ij} C_{ij}(M)$ is the Frobenius decomposition of $S_i(M)$ and that $N_i(M)^{\eta_i} = \sum_{j=1}^{\rho_i} (M- \lambda_{ij} I_n)^{\eta_i} C_{ij}(M) =0$, because $\sum_{j=1}^{\rho_i} (X- \lambda_{ij})^{\eta_i} C_{ij}(X)$ is multiple of the minimal polynomial $m(X)$ of $M$. Moreover each $N_i(M)$ is a polynomial expression of $M$ of degree at most $deg (m(X))$.
\end{notas}

\begin{prop}\label{propr-Si-Ni} a) For every $i$, 
$S_i(X)$ and $N_i(X)$ have coefficients in $\mathbb{K}^\dag$, 

$deg(S_i(X)) < deg(m(X))$, $deg(N_i(X)) \le deg(m(X))$ and
$$S(M) = \sum_{i=1}^{r'} S_i(M), \ \ \ N(M) = \sum_{i=1}^r N_i(M).$$

b) For all admissible indices $i, h, l$, if $u \in Im\, C_{hl}(M)$, we have: $S_i(M)u= \delta_{ih} \lambda_{hl} u$ and $N_i(M)u= \delta_{ih} N(M)u$.

c) $S_i(M) S_j(M) =0$ as soon as $i \ne j$.

d) $Ker \, S(M) = Ker\, M^{\eta_r}$.

e) If $r \ge 2$, the minimal polynomial of each $S_i(M)$ is $X g_i(X)$.
\end{prop}

\begin{proof}
The polynomials have coefficients in $\mathbb{K}^\dag$: indeed, arguing as in \ref{SN-dec}, they
are $Aut(\mathbb{F}/\mathbb{K})$-invariant. 
The inequalities on the degrees follow from the degrees of $C_{ij}(X)$'s  (remember \ref{ident-bezout-per-M}).
The additive decompositions of $S(M)$ and $N(M)$ follow from the definitions. This concludes (a).

We get (b) by multiplying $S_i(M)$ and $N_i(M)$ on the right with an element of the form $u = C_{hl}(M) v$ and by remembering the definitions and the properties of the involved matrices.
Again a direct computation allows to get (c).

Assertion (d) is trivial, when $M$ is nonsingular.
Otherwise we have: $\lambda_{r1}=0$, $\rho_r=1$ and $r'=r-1$.
In this case we want to prove that there exists a matrix $W$  such that $S(M) = W M^{\eta_r}$. This implies that $Ker \, M^{\eta_r} \subseteq Ker \, S(M)$.

We have: 

$m(X) = [\, \prod_{h=1}^{r-1} \prod_{l=1}^{\rho_h}(X-\lambda_{hl})^{\eta_h}] X^{\eta_r}$, 
so 
$G_{ij}(X) = \dfrac{\prod_{h=1}^{r-1} \prod_{l=1}^{\rho_h}(X-\lambda_{hl})^{\eta_h}}{(X-\lambda_{ij})^{\eta_i}} X^{\eta_r}$. 

Since the fractions $W_{ij}(X) = \dfrac{\prod_{h=1}^{r-1} \prod_{l=1}^{\rho_h}(X-\lambda_{hl})^{\eta_h}}{(X-\lambda_{ij})^{\eta_i}}$ are actually polynomials for every $i \le r-1$, we can conclude that $G_{ij}(M) = W_{ij}(M) M^{\eta_r}$.

Hence we get the desired assertion with $W= \sum_{i=1}^{r-1} \sum_{j=1}^{\rho_i} \lambda_{ij}B_{ij}(M) W_{ij}(M)$.

For the opposite inclusion we have: $M C_{r1}(M) = S(M) C_{r1}(M) + N(M) C_{r1}(M)$ and so $M C_{r1}(M) = N_r(M) C_{r1}(M)$. Hence $M^{\eta_r} C_{r1}(M) = (M C_{r1}(M))^{\eta_r} = N_r(M)^{\eta_r} C_{r1}(M)=0$ since $N_r(M)$ is nilpotent of order $\eta_r$. We can conclude because $Im \, C_{r1}(M) = Ker \, S(M)$, by \ref{Propr-partizione}.

Finally \ref{Propr-partizione} implies that, for every index $i$, the set of all eigenvalues of $S_i(M)$ is $\{ \lambda_{i1}, \cdots , \lambda_{i \rho_i}, 0 \}$, because $r \ge 2$, and that $S_i(M)$ is semisimple, so its minimal polynomial is $X g_i(X)$. This concludes (e).
\end{proof}

\begin{rem}\label{S_i-unbreakable}
If $S(M) \in M_n(\mathbb{K}^\dag) \setminus \{0\}$ is the (nonzero) semisimple part of $M$, then $S_1(M), \cdots , S_{r'}(M) \in M_n(\mathbb{K}^\dag)  \setminus \{ 0\}$ are the unbreakable semisimple components of $S(M)$. Each matrix $S_i(M)$ is a polynomial expression of $M$ of degree at most $deg(m(X)) -1$.
This fact follows directly by \ref{esistenza}.
\end{rem}

\begin{prop}\label{N_i-uniche}
Let $N(M)$ be the nilpotent part of $M$. Then the matrices  $N_1(M), \cdots , N_r(M)$ are in  $M_n(\mathbb{K}^\dag)$ and are uniquely determined in $M_n(\overline{\mathbb{K}})$ by the conditions:

a) $N(M) = N_1(M) + \cdots + N_r(M)$;

b) $N_h(M) S_l(M) = 0$ as soon as $h \ne l$;

c) for every $h=1, \cdots , r'$, there is a matrix \, $W_h \in M_n(\overline{\mathbb{K}})$ such that

$N_h(M) = W_h M^{\eta_r}$.
\end{prop}

\begin{proof}
Standard computations show that the matrices $N_h(M)$'s satisfy (a) and (b).

Part (c) is trivial, if $M$ is nonsingular. Otherwise, as in the proof of \ref{propr-Si-Ni} (d), for every $h \le r'$, we have $G_{hj}(M) = W_{hj}(M) M^{\eta_r}$. 

Hence, after posing $W_h= \sum_{j=1}^{\rho_h} (M-\lambda_{hj} I_n) B_{hj}(M) W_{hj}(M)$, we get $N_h(M) = W_h M^{\eta_r}$ and this concludes (c).

Now let $N_1, \cdots, N_r $ be matrices in $M_n(\overline{\mathbb{K}})$ satisfying (a), (b), (c).

From the decomposition $\overline{\mathbb{K}}^n = \oplus_{h=1}^r \oplus_{l=1}^{\rho_h} Im \, C_{hl}(M)$ of \ref{Propr-partizione} and \ref{propr-Si-Ni} (b), 
it suffices to check that $N_iv = \delta_{ih} N(M)v$ for every $v \in Im \, C_{hl}(M)$ and for all admissible indices  $h, i, l$. 

Assume first that $h \le r'$, so $\lambda_{hl} \ne 0$. Then $\delta_{ih} N(M) v =\delta_{ih} N(M) S_h(M) \dfrac{v}{\lambda_{hl}} = \delta_{ih}[N_1  S_h(M) \dfrac{v}{\lambda_{hl}} + \cdots +  N_r S_h(M) \dfrac{v}{\lambda_{hl}}] = \delta_{ih} N_h S_h(M) \dfrac{v}{\lambda_{hl}} = N_i  S_h(M) \dfrac{v}{\lambda_{hl}} = N_i v$ 
as requested, by condition (b). This completes the proof in case of $r'=r$.

Assume now that $r' = r-1$, so $\rho_r = 1$, $\lambda_{r1} =0$  and $Im \, C_{r1}(M) = Ker\, S(M) = Ker (M^{\eta_r})$ by \ref{pol-min} and \ref{propr-Si-Ni} (d).

Let $v \in Im \, C_{r1}(M)$ i.e. $M^{\eta_r} v= 0$. The condition (c) gives that $N_i (M)v = N_i v =0$ for every $i \le r'=r-1$, hence $N_i = N_i(M)$ for every $i \le r-1$. Since $N_1 + \cdots + N_r = N_1(M) + \cdots + N_r(M)$, we get also that $N_r = N_r(M)$, as requested.
\end{proof}

\begin{remdef}\label{FineSN}
The previous results assert the existence and the uniqueness of the decompositions 
$S(M) = S_1(M) + \cdots + S_{r'}(M)$ and 
$N(M) = N_1(M) + \cdots + N_r(M)$. 

We call the consequent decomposition 
\begin{center}
$M = S_1(M) + \cdots + S_{r'}(M) + N_1(M) + \cdots + N_r(M)$
\end{center}
\emph{the fine Jordan-Chevalley decomposition} (or \emph{fine SN decomposition}) \emph{of} $M$ and the matrices $S_i(M)$'s and $N_h(M)$'s \emph{the fine components} respectively of $S(M)$ and of $N(M)$.

When $0$ is an eigenvalues of $M$ (so $r'=r-1$, $\rho_r = 1$ and $\lambda_{r1}= 0$), we agree that also the null matrix $S_r(M)= \lambda_{r1} C_{r1}(M) =0$ is a fine component of $S(M)$. Hence the fine components of $S(M)$ and of $N(M)$ are always $r$. This agreement will allow to simplify the language and the statements of the next section. Indeed every $S_i(M)$ and every $N_i(M)$ corresponds to a conjugacy class over $\mathbb{K}$ of eigenvalues of $M$.
\end{remdef}

\section{A Schwerdtfeger-type formula}

\begin{remdef}\label{val-ass}
a) An \emph{absolute value} over $\mathbb{K}$ is a function $|.|: \mathbb{K} \mapsto \mathbb{R}$ $x \to |x|$ such that

$|x| \ge 0$ for every $x \in \mathbb{K}$ and $|x|=0$ if and only if $x=0$;

$|x+y| \leq |x| + |y|$ for every $x,y$;

$|xy|=|x||y|$ for every  $x,y$.

We call such a pair $(\mathbb{K}, |.|)$ \emph{a valued field}.
We refer for instance to \cite{War1989} Ch.III, Ch.IV and to \cite{Lang2002} Ch.XII for more information.

In particular we recall that we can define an absolute value over every field, by putting $|x|=1$ for every $x \ne 0$, this is called \emph{trivial absolute value}.

The absolute value of a valued field always extends in a unique way to its \emph{completion} (with absolute value denoted again by $|.|$). Therefore it is not restrictive to assume that the valued field is already complete. Moreover if the absolute value is not trivial and the valued field is complete, then it extends in a unique way to its algebraic closure. We denote the extended absolute value by the same notation.

Finally a non-trivial absolute value over a complete valued field $\mathbb{K}$ is constant on every conjugacy class over $\mathbb{K}$ (see for instance \cite{Lang2002} Ch.XII Prop. 2.6).

b) Let $(\mathbb{K}, |.|)$ be a complete valued field endowed with a non-trivial absolute value. 
We can consider on $M_n(\overline{\mathbb{K}})$ any norm, $\| . \|$, which is \emph{compatible with the absolute value}, i.e. $\|\lambda A \|= |\lambda| \| A\|$ for every $\lambda \in \overline{\mathbb{K}}$ and every $A \in M_n(\overline{\mathbb{K}})$. 

The restriction of this norm to $M_n(\mathbb{K})$ is equivalent to every other norm over $M_n(\mathbb{K})$ and induces the product topology of $M_n(\mathbb{K})$, viewed as a product space (see for instance \cite{Lang2002} Ch.XII Prop.2.2). Hence $M_n(\mathbb{K})$ is a complete metric space.

If the above norm over $M_n(\overline{\mathbb{K}})$ is \emph{submultiplicative} (i.e. $\|AB|| \le \|A \|\ \,\|B\|$ for every pair of matrices $A, B$), then, arguing as in \cite{HoJ2013} pp.347--348, standard computations show that if $\lambda \in \overline{\mathbb{K}}$ is any eigenvalue of $A \in M_n(\overline{\mathbb{K}})$, then $|\lambda| \le \| A \|$ and moreover it is possible to prove that  the \emph{spectral radius} of $A$ is  

$\rho(A) = inf \{ \| A \|' \ / \  
\| . \|' \mbox{ is a submultiplicative norm on } M_n(\overline{\mathbb{K}})
\mbox{ compatible with } |.| \}$.

c) Let $(\mathbb{K}, |.|)$ be a complete valued field endowed with a non-trivial absolute value.
Let $f(X)=\sum_{m=0}^\infty a_m X^m$, $a_m \in \mathbb{K}$ be a series, to which we can associate the real series $\sum_{m=0}^\infty |a_m| X^m$, whose radius of convergence, $R_f \in \mathbb{R} \cup \{+ \infty\}$, is the supremum of the real numbers $t \ge 0$ such that $|a_m| t^m$ is upper bounded. 

We call $R_f$ the \emph{radius of convergence of} $f$.

Now let $\Omega_{f, \mathbb{K}}$ be the set of matrices $A \in M_n(\mathbb{K})$ such that $\rho(A) < R_f$. 

We remark that $\Omega_{f, \mathbb{K}}$ can be characterized as the set of matrices $A \in M_n(\mathbb{K})$ such that there exists a submultiplicative norm $\|.\|$ on $M_n(\overline{\mathbb{K}})$, compatible with $|.|$,  such that $\|A \| < R_f$.

Moreover $\Omega_{f, \mathbb{K}}$ is an open subset of $M_n(\mathbb{K})$ and, if $M \in \Omega_{f, \mathbb{K}}$, then both semisimple and nilpotent parts of $M$ and their related fine components belong to $\Omega_{f, \mathbb{K}}$.

Again, if $M \in \Omega_{f, \mathbb{K}}$, then the series $f(M)$ converges to a matrix in $M_n(\mathbb{K})$, being this last complete.
Moreover if $\lambda$ is an eigenvalue of $M$, then $\lambda$ is in the splitting field, $\mathbb{F}$, of the minimal polynomial of $M$, which is complete, because it is a finite extension of $\mathbb{K}$ (see for instance \cite{Lang2002} Ch.XII Prop.2.5) and $\lambda \in D_f = \{ \alpha \in \mathbb{F} \ / \ |\alpha| < R_f \}$. Hence $f(\lambda)$ is a convergent series to an element of $\mathbb{F}$.

d) We can write any polynomial $f(X)$, having coefficients in any field $\mathbb{K}$, as series  with infinite null coefficients. In this case we agree that the \emph{radius of convergence of} $f$ is $R_f = + \infty$, so $\Omega_{f, \mathbb{K}} = M_n(\mathbb{K})$ and $D_f = \mathbb{F}$.

e) Let $f(X)=\sum_{m=0}^\infty a_m X^m$ be either a polynomial over any field $\mathbb{K}$ or a series, having coefficients in a complete valued field $\mathbb{K}$ endowed with a non-trivial absolute value. 

We denote by $R_f$ its radius of convergence and by $\Phi_k(X)$ the functions of the same type of $f(X)$, given by

$
\Phi_k(X) = \sum_{m=k}^\infty {m \choose k} a_m X^{m-k}
$,
where ${m \choose k} = \overbrace{1 + \cdots +1}^{{m \choose k} \mbox{ times }}$  ($1$ the unity in $\mathbb{K}$).

Standard computations show that

i) the radius of convergence of every $\Phi_k(X)$ is at least $R_f$ and $\Phi_0(X) = f(X)$;

ii) if $char(\mathbb{K})=0$, then $\Phi_k(X) = \dfrac{1}{k!}\dfrac{d^k}{dX^k}f(X)$ for every $k \ge 1$;

iii) if $char(\mathbb{K})$ is positive, then $\Phi_k(X) = \dfrac{1}{k!}\dfrac{d^k}{dX^k}f(X)$ for every $1 \le k < char(\mathbb{K})$,

where the $k$-th derivative denotes the series obtained by differentiating $k$ times term-by-term.

With same notations and the same arguments as in (c) and (d), by (i) we get that $\Phi_k(\lambda)$ is a convergent series in $\mathbb{F}$ , for every $k \ge 1$ and every $\lambda \in D_f$.
\end{remdef}

\begin{prop}\label{Schw-Sylv}
With the same notations as in \ref{pol-min} and in \ref{val-ass}, let $f(X)$ be either a polynomial over any $\mathbb{K}$ or a series over $\mathbb{K}$ supposed to be a complete valued field with respect to a non-trivial absolute value.

If $M \in \Omega_{f, \mathbb{K}}$ and $\lambda$ is an eigenvalue of $M$ (so it belongs to the splitting field $\mathbb{F}$ of the minimal polynomial of $M$), then $f(M) \in M_n(\mathbb{K})$ and $\Phi_k(\lambda) \in \mathbb{F}$ , for every $k \ge 0$.

Furthermore:
$$f(M) = \sum_{i=1}^r \sum_{j=1}^{\rho_i}\, [ \sum_{k=0}^{\eta_i -1} \Phi_k(\lambda_{ij})(M-\lambda_{ij} I_n)^k] \, C_{ij}(M).$$
The semisimple and the nilpotent parts of $f(M)$ are respectively:

$S(f(M)) = \sum_{i=1}^r \sum_{j=1}^{\rho_i} f(\lambda_{ij}) C_{ij}(M)$ and

$N(f(M)) = \sum_{i=1}^r \sum_{j=1}^{\rho_i} \sum_{k=1}^{\eta_i} \Phi_k(\lambda_{ij})(M-\lambda_{ij} I_n)^k C_{ij}(M)$.
\end{prop}

\begin{proof}
The first part of the statement has been already proved in \ref{val-ass}.

To complete the proof and to simplify the notations, we agree that $\lambda_{r1}^s = 0$ for every integer $s \le 0$ when  $\lambda_{r1} =0$.

Remembering the properties of the $C_{ij}(M)$'s in \ref{ident-bezout-per-M}, since 

$M= \sum_{i=1}^r \sum_{j=1}^{\rho_i}(\lambda_{ij} I_n + (M-\lambda_{ij} I_n)) C_{ij}(M)$, 

we have: $M^m = \sum_{i,j} [\sum_{k=0}^m {m \choose k} \lambda_{ij}^{m-k} (M-\lambda_{ij} I_n)^k] C_{ij}(M)$.

After posing ${m \choose k} =0$ for $k > m$, so $\Phi_k(X) = \sum_{m=1}^\infty {m \choose k} a_m X^{m-k}$, we can write: 

$M^m = \sum_{i,j} [\sum_{k=0}^\infty {m \choose k} \lambda_{ij}^{m-k} (M-\lambda_{ij} I_n)^k] C_{ij}(M)$.

Therefore: 
\begin{align*}
&f(M) -f(0)I_n = \sum_{m=1}^\infty a_m M^m 
= \sum_{k=0}^\infty \sum_{i,j}  [\sum_{m=1}^\infty {m \choose k} a_m \lambda_{ij}^{m-k}(M-\lambda_{ij}I_n)^k]C_{ij}(M)\\
&= \sum_{i,j}   [\sum_{m=1}^\infty a_m \lambda_{ij}^m]C_{ij}(M) + \sum_{k=1}^\infty \sum_{i,j}   [\sum_{m=1}^\infty {m \choose k} a_m \lambda_{ij}^{m-k}] \, (M-\lambda_{ij}I_n)^k C_{ij}(M) \\
&=\sum_{i,j} (\Phi_0(\lambda_{ij}) -f(0)) C_{ij}(M) + \sum_{i,j} [\sum_{k=1}^\infty \Phi_k(\lambda_{ij})(M-\lambda_{ij}I_n)^k] C_{ij}(M) \\
&= \sum_{i,j} [\sum_{k=0}^\infty \Phi_k(\lambda_{ij})(M-\lambda_{ij}I_n)^k] C_{ij}(M) - f(0)I_n.
\end{align*}

Hence:
$
f(M) = \sum_{i,j} [\sum_{k=0}^{\eta_i -1} \Phi_k(\lambda_{ij})(M-\lambda_{ij}I_n)^k] C_{ij}(M),
$

because $(M-\lambda_{ij}I_n)^{\eta_i} G_{ij}(M)=0$ by definition of $G_{ij}$ in \ref{pol-min}.

We conclude by remarking that for $k=0$ we get the semisimple part, while the remaining part is the nilpotent one.
\end{proof}

\begin{rem}
With the same notations as above if $M$ is semisimple then 
$$f(M)  = \sum_{i=1}^r \sum_{j=1}^{\rho_i} f(\lambda_{ij}) C_{ij}(M).$$
Indeed, if $M$ is semisimple, then $\eta_i = 1$ for every $i$.

In real and complex cases the above formula reduces to Sylvester's formula, while the more general formula
$$f(M) = \sum_{i=1}^r \sum_{j=1}^{\rho_i} \sum_{k=0}^{\eta_i -1} \Phi_k(\lambda_{ij})(M-\lambda_{ij} I_n)^k C_{ij}(M)$$ 
reduces to Schwerdtfeger's formula (see for instance \cite{HoJ1991}) Ch. 6).

Note that the first formula, Proposition \ref{Schw-Sylv} and the formula for $S(M)$ in Theorem \ref{SN-dec}, give that $S(f(M)) = f(S(M))$ for every $M \in \Omega_{f, \mathbb{K}}$.
\end{rem}

\begin{rem}\label{sketch-decomp-fine}
Let $f(X)$ and $M$ be as in \ref{Schw-Sylv}. 

While $S(M)= \sum_{i=1}^{r'} S_i(M) = \sum_{i=1}^{r'} \sum_{j=1}^{\rho_i} \lambda_{ij} C_{ij}(M)$ gives both the fine Jordan-Cheval\-ley decomposition and the Frobenius decomposition of $S(M)$,  from the expression  
$\sum_{i=1}^r \sum_{j=1}^{\rho_i} f(\lambda_{ij}) C_{ij}(M)$,
we cannot directly deduce the analogous decompositions of $S(f(M))$; nevertheless both decompositions can be deduced from it.

For the  Frobenius decomposition of $S(f(M))$, if the $f(\lambda_{ij})$'s are not pairwise distinct, we can sum the different $C_{ij}(M)$'s with the same coefficients $f(\lambda_{ij})$'s to get the desired Frobenius covariants as suitable sums of the Frobenius covariants of $M$.

For the fine Jordan-Chevalley decomposition of $S(f(M))$ we remember  that each fine component of the semisimple part of a matrix corresponds to a conjugacy class over $\mathbb{K}$ of eigenvalues of the matrix.

Since $\mathbb{F}$ is a finite normal extension of $\mathbb{K}$, the conjugacy class over $\mathbb{K}$ of every element of $\mathbb{F}$ is contained in $\mathbb{F}$ and moreover the conjugacy class over $\mathbb{K}$ of every element of $D_f$ is contained in $D_f$.

$Aut(\mathbb{F}/\mathbb{K})$ acts transitively over each conjugacy class over $\mathbb{K}$ contained in $\mathbb{F}$ as observed in \ref{pol-min} (d).

Since $f$ commutes with every element of $Aut(\mathbb{F}/\mathbb{K})$, $f$ maps conjugacy classes over $\mathbb{K}$, contained in $D_f$, onto conjugacy classes over $\mathbb{K}$, contained in $\mathbb{F}$. However different conjugacy classes can be mapped by $f$ into the same conjugacy class. 

If two such conjugacy classes are mapped by $f$ onto the same conjugacy class, we say that the corresponding fine components of $S(M)$ are $f$-\emph{equivalent}.

Therefore every fine component $S_h(f(M))$ of $S(f(M))$ is sum of terms of the type 
$\sum_{j=1}^{\rho_i} f(\lambda_{ij}) C_{ij}(M)$, where the sum is extended to all indices, $i$, corresponding to the fine components of $S(M)$ of a given $f$-equivalence class.

Analogously every fine component $N_h(f(M))$ of $N(f(M))$ is sum of terms of the type $\sum_{j=1}^{\rho_i} \sum_{k=1}^{\eta_i} \Phi_k(\lambda_{ij})(M-\lambda_{ij} I_n)^k C_{ij}(M)$, where again the sum is extended to all indices, $i$, corresponding to the fine components of $S(M)$ of a given $f$-equivalence class.
\end{rem}

\section{Some consequences on real closed fields}

\begin{remdef}
a) A field $\mathbb{K}$ is said to be a \emph{real closed field} if it can be ordered as field and no proper algebraic extension of $\mathbb{K}$ can be ordered as field.

Of course any ordered field has characteristic $0$.

It is known that if $\mathbb{K}$ is a real closed field, then it has a unique order (as field) and that two equivalent characterizations of being a real closed field are:

i) $\sqrt{-1} \notin \mathbb{K}$ and $\mathbb{K}(\sqrt{-1})$ is algebraically closed;

ii) $\mathbb{K}$ admits an order as field such that its positive elements have square root in $\mathbb{K}$ and any polynomial of odd degree in $\mathbb{K}[X]$ has a root in $\mathbb{K}$.

Note that, as in the ordinary real case, $\overline{\mathbb{K}} = \mathbb{K}(\sqrt{-1})$ and the irreducible polynomials in $\mathbb{K}[X]$ have degree at most $2$. Moreover if $\lambda = a + b \sqrt{-1}$ is root of $h(X) \in \mathbb{K}[X]$, then also its \emph{conjugate} $\overline{\lambda}=  a - b \sqrt{-1}$ is root of $h(X)$.

We refer for instance to \cite{Lang2002} Ch.XI \S2 and to \cite{Raw1993} Ch.15 for more information.

b) Let $\mathbb{K}$ be a real closed field. The $\mathbb{K}$-\emph{norm} of an element $\lambda=a + b\sqrt{-1} \in \overline{\mathbb{K}} = \mathbb{K}(\sqrt{-1})$ is the unique positive square root of $a^2 + b^2 \in  \mathbb{K}$, we denote by $N_\mathbb{K}(\lambda)$.

The norm is strictly positive as soon as $\lambda \ne 0$.

Standard computations show that, as in ordinary real case, every element $\lambda \in \overline{\mathbb{K}} \setminus \{ 0\}$ can be written as $\lambda = N_\mathbb{K}(\lambda) \dfrac{\lambda}{N_\mathbb{K}(\lambda)}$, where $N_\mathbb{K}(\lambda)$ is a strictly positive element of $\mathbb{K}$ and $N_\mathbb{K}(\dfrac{\lambda}{N_\mathbb{K}(\lambda)})=1$.
\end{remdef}

\begin{prop}\label{deltasigmau}
Assume that $\mathbb{K}$ is a real closed field. Let $M = S(M) + N(M) \in M_n(\mathbb{K})$ be a matrix  with its additive Jordan-Chevalley decomposition.

a) The Frobenius decomposition of $S(M)$ is
$$
S(M) = \sum_{h=1}^{s_1} [\lambda_{h1} \, C_{h1}(M) + \overline{\lambda_{h1}} \, C_{h2}(M)] + \sum_{i=s_1 + 1}^{s_1+s_2} \gamma_i C_{i1}(M) - \sum_{i= s_1+s_2 +1}^{r'} \gamma_i C_{i1}(M),
$$
where the $\gamma_i$'s are strictly positive elements of $\mathbb{K}$ for every $i = s_1 +1, \cdots , r'$, the $\lambda_{h1}$'s are in $\overline{\mathbb{K}} \setminus \mathbb{K}$ and $C_{h2}(M) = \overline{C_{h1}(M)}$ for every $h = 1, \cdots , s_1$.

b) Moreover, if $M \in GL_n(\mathbb{K})$, then there is a unique way to write
\begin{center}
$
M = \Delta \Sigma U
$
\end{center} 
as product of three mutually commuting matrices with coefficients in $\mathbb{K}$, with
$\Delta$ diagonalizable over $\mathbb{K}$ and strictly positive eigenvalues, $\Sigma$ semisimple and eigenvalues of norm $1$ and $U$ unipotent.

In particular we have:
\begin{align*}
&\Delta =\sum_{h=1}^{s_1} N_{\mathbb{K}}(\lambda_{h1}) [C_{h1}(M) + C_{h2}(M)] + \sum_{i=s_1 + 1}^{r'} \gamma_i C_{i1}(M),\\
&\Sigma = \sum_{h=1}^{s_1} [\dfrac{\lambda_{h1}}{N_{\mathbb{K}}(\lambda_{h1})} \, C_{h1}(M) + \dfrac{\overline{\lambda_{h1}}}{N_{\mathbb{K}}(\lambda_{h1})} C_{h2}(M)]
+ \sum_{i=s_1 + 1}^{s_1+s_2} C_{i1}(M) - \sum_{i= s_1+s_2 +1}^{r'} C_{i1}(M)\\
&(\mbox{hence \ } S(M)= \Delta \Sigma) \mbox{ and \ } U = I_n + S(M)^{-1}N(M).
\end{align*}

The matrices $\Delta$ , $\Sigma$ and $U$ are polynomial functions of $M$.
\end{prop}

\begin{proof}
Part (a) follows from \ref{SN-dec} and from the fact that the irreducible polynomials in $\mathbb{K}[X]$ have degree at most $2$.

For (b): by \ref{JC-moltiplicativa} we have $M=S(M) U(M) = S(M) (I_n+ S(M)^{-1}N(M))$ (with all factors which are polynomials in $M$), where $S(M)$ has the expression in (a).
Now, for every $h$, we write: $\lambda_{h1} = N_\mathbb{K}(\lambda_{h1}) \dfrac{\lambda_{h1}}{N_\mathbb{K}(\lambda_{h1})}$, where $N_\mathbb{K}(\lambda_{h1})$ is a strictly positive element of $\mathbb{K}$ and $N_\mathbb{K}(\dfrac{\lambda_{h1}}{N_\mathbb{K}(\lambda_{h1})})=1$ ; so, by standard computations, we get the equality $M = \Delta \Sigma U$ with $\Delta, \Sigma, U$ satisfying the requested properties.

For the uniqueness, assume that $M =  \Delta' \Sigma' U'$ is another decomposition with the expected properties. Since $\Delta', \Sigma', U'$ are pairwise commuting, each one commutes with $M$ and so with any polynomial expression of $M$ (as $\Delta, \Sigma$ and $U$).
Moreover $\Delta' \Sigma'$ is semisimple, so from the uniqueness of the multiplicative Jordan-Chevalley decomposition: $U=U'$ and $\Delta \Sigma = \Delta' \Sigma'$. Now $\Delta^{-1} \Delta' = \Sigma \, \Sigma'^{-1}$. By commutativity, the left side is a diagonalizable matrix with strictly positive eigenvalues, while the right side is a semisimple matrix with eigenvalues of norm $1$. Since the unique positive element of $\mathbb{K}$ with norm $1$ is $1$ itself , both products are the identity matrix.
\end{proof}

\begin{remdef}
If $\mathbb{K} = \mathbb{R}$, the decomposition $M= \Delta  \Sigma U$ in \ref{deltasigmau} (b) is well-known (see for instance  \cite{Helg2001} pp. 430--431). Hence, following the usual terminology, we refer to $M= \Delta \Sigma U$ as \emph{the complete multiplicative Jordan-Chevalley decomposition of} $M$ also in case of any real closed field.

Note that, when $\mathbb{K} = \mathbb{R}$, part (b) of the previous Proposition implies that every matrix of $GL_n(\mathbb{R})$ can be written in a unique way as product of a real matrix similar to a positive definite symmetric matrix, of a real matrix similar to an orthogonal matrix and of a real unipotent matrix, where the three matrices are pairwise commuting. 
\end{remdef}

\begin{defi}
Let $\mathbb{K}$ be real closed, so $\overline{\mathbb{K}} = \mathbb{K}(\sqrt{-1})$. As in the ordinary real case we say that a matrix $A \in M_n(\overline{\mathbb{K}})$ is \emph{normal} (respectively \emph{hermitian}) if $AA^* = A^*A$ (respectively $A= A^*$) where $A^*$ is the transpose conjugated matrix of $A$ (if $A \in M_n(\mathbb{K})$, then $A^*$ is simply the transpose of $A$).
\end{defi}

\begin{rem}
As in the ordinary real and complex cases we can define a positive definite hermitian product over $\overline{\mathbb{K}}^n$ by $<z, w>_{\overline{\mathbb{K}}^n} \, = z^* w$ and a positive definite scalar product  over $\mathbb{K}^n$ by $<z, w>_{\mathbb{K}^n} \, = z^T w$ for all $z, w$ (column) vectors in $\overline{\mathbb{K}}^n$ and of $\mathbb{K}^n$ respectively.

As noted in \cite{Lang2002} p.585, the ordinary spectral theorems are valid if $\mathbb{K}$ is real closed. Hence a matrix $A \in M_n(\overline{\mathbb{K}})$ is normal if and only if there exists an orthonormal basis of $\overline{\mathbb{K}}^n$ of eigenvectors of $A$ and a matrix $A \in M_n(\mathbb{K})$ is symmetric if and only if there exists an orthonormal basis of $\mathbb{K}^n$ of eigenvectors of $A$.
\end{rem}

\begin{lemma}\label{normal-matrices}
Let $\mathbb{K}$ be a real closed field and $A \in M_n(\overline{\mathbb{K}}) \setminus \{0\}$. Then

i) $A$ is normal if and only if it is semisimple and its Frobenius covariants are hermitian matrices;

ii) $A$ is hermitian if and only if it is semisimple, its Frobenius covariants are hermitian matrices and its eigenvalues are in $\mathbb{K}$.
\end{lemma}

\begin{proof}
Let $\lambda_1, \cdots, \lambda_s$ be the nonzero distinct eigenvalues of the normal matrix $A \ne 0$ of multiplicity $n_1, \cdots , n_s$ respectively and choose a set orthonormal (column) eigenvectors $v_{ij}$, $1 \le i \le s$, $1 \le j \le n_i$, such that every $v_{ij}$ is an eigenvector associated to the eigenvalue $\lambda_i$.
It is easy to check that $A = \sum_{i=1}^s \sum_{j=1}^{n_i} \lambda_i v_{ij}v_{ij}^* = \sum_{i=1}^s \lambda_i A_i$, where every $A_i = \sum_{j=1}^{n_i}v_{ij}v_{ij}^* = A_i^*$ is a nonzero hermitian matrix in $M_n(\overline{\mathbb{K}})$ and $A_i A_h =  \delta_{ih} A_i$. So by \ref{unicita-dec-Frob} the matrices $A_i$'s are the Frobenius covariants of $A$.

For the converse, if $A= \sum_{i=i}^s \lambda_i A_i$ is the Frobenius decomposition of $A$ with $A_i$'s hermitian matrices, then $A^*= \sum_{i=i}^s \overline{\lambda}_i A_i$ and so $A A^* = A^* A = \sum_{i=1}^s N_{\mathbb{K}}(\lambda_i)^2 A_i$.

For part (ii), an implication follows by remarking that hermitian matrices are also normal and that their eigenvalues are in $\mathbb{K}$. The other implication follows directly from the properties of the Frobenius decomposition.
\end{proof}

\begin{defi}
A non-empty family of matrices $A_1, \cdots , A_p \in M_n(\overline{\mathbb{K}}) \setminus \{ 0 \} $ is said to be an \emph{SVD system}, if 

$A_i^* A_j = A_i A_j^* = 0$ for every $i \ne j$;

$A_i A_i^* A_i = A_i$ for every $i$.

We call \emph{singular value decomposition of} $A \in M_n(\overline{\mathbb{K}})$ (shortly \emph{SVD}) any decomposition  
$$A= \sum_{i=1}^p \sigma_i A_i  ,$$
where $A_1, \cdots , A_p \in M_n(\overline{\mathbb{K}}) \setminus \{0\}$ form an SVD system and $\sigma_1 > \cdots > \sigma_p >0$ are elements of $\mathbb{K}$.
\end{defi}

\begin{prop}\label{SVD-esiste}
Assume that $\mathbb{K}$ is a real closed field. 

Then every matrix $A \in M_n(\overline{\mathbb{K}}) \setminus\{0\}$ has an SVD: $A= \sum_{i=1}^p \sigma_i A_i$.

If $A \in M_n(\mathbb{K}) \setminus\{0\}$, then we can take every $A_i$ in $M_n(\mathbb{K})$.
\end{prop}

\begin{proof}
The matrix $A^* A$ is hermitian positive semidefinite with non-negative eigenvalues in $\mathbb{K}$. Note that $Ker(A^* A)= Ker(A)$. 
Indeed, if $w \in Ker \, A^* A$, then $0 = <A^* A w, w>_{\overline{\mathbb{K}}^n} = <A w, A w>_{\overline{\mathbb{K}}^n}$, so $Aw=0$ and $w \in Ker \, A$. The other inclusion is trivial. Hence $A \ne 0$ implies $A^* A \ne 0$ and so there exists a nonzero eigenvalue of $A^* A$.
Up to reordering, we can assume that the strictly positive eigenvalues are $\lambda_1 > \cdots > \lambda_p$ (of multiplicity $n_1, \cdots , n_p$ respectively), with $p \ge 1$, and we denote by $\sigma_i$ the unique strictly positive square root of $\lambda_i$, for $i = 1 , \cdots  , p$.

As in \ref{normal-matrices}, we consider a set of orthonormal (column) eigenvectors of $A^* A$ given by $\{ v_{ij} \ / \ 1 \le i \le p, \ 1 \le j \le n_i \}$ and, if necessary, we complete it to an orthonormal basis of $\overline{\mathbb{K}}^n$ by means of an orthonormal basis  $\{w_1, \cdots , w_\nu \}$ of $Ker(A^* A$). 

Since these vectors form an orthonormal basis, we have $I_n = \sum_{i,j} v_{ij} v_{ij}^* + \sum_l w_l w_l^*$.

Therefore $A= \sum_{i,j} A v_{ij} v_{ij}^* = \sum_i \sigma_i \sum_j \dfrac{A v_{ij} v_{ij}^*}{\sigma_i} = \sum_{i=1}^p \sigma_i A_i$ 
with 

$A_i = \sum_{j=1}^{n_i} \dfrac{A v_{ij} v_{ij}^*}{\sigma_i}$. It is easy to check that this an SVD decomposition of $A$.

Finally, if $A \in M_n(\mathbb{K}) \setminus \{0\}$, we can consider $A^T A$ and the corresponding orthonormal basis of $\mathbb{K}^n$. Hence analogously we get: $A= \sum_{i,j} A v_{ij} v_{ij}^T = \sum_i \sigma_i \sum_j \dfrac{A v_{ij} v_{ij}^T}{\sigma_i} = \sum_{i=1}^p \sigma_i A_i$ with $A_i = \sum_{j=1}^{n_i} \dfrac{A v_{ij} v_{ij}^T}{\sigma_i} \in M_n(\mathbb{K})$.
\end{proof}

\begin{remdef}
Note that in the previous proof the coefficients $\sigma_i$'s are the positive square roots of the nonzero eigenvalues of $A^* A$ (or of $A^T A$ if $A \in M_n(\mathbb{K})$). As in real and complex cases,  we call them \emph{singular values} of the matrix $A$ (see for instance \cite{HoJ2013} Thm.2.6.3 and \cite{OttPaol2015} Thm. 3.4).

If $A$ is normal, then, as in the proof of \ref{normal-matrices}, we get that the singular values of $A$ are the distinct elements of the form $N_{\mathbb{K}}(\lambda_i)$, where $\lambda_i$ runs over the set of nonzero eigenvalues of $A$.

The above decomposition is unique, as precised in the following
\end{remdef}

\begin{prop}\label{SVF-unica}
Let $\mathbb{K}$ be a real closed field and $A = \sum_{i=1}^p \sigma_i A_i \in M_n(\overline{\mathbb{K}}) \setminus \{0\}$ be a matrix with its SVD constructed in \ref{SVD-esiste}.
Let $A = \sum_{i=1}^q \tau_i B_i$ be any other SVD of $A$.
Then $q=p$ and, for every $i$, $\tau_i= \sigma_i$ and $B_i = A_i$.
\end{prop}

\begin{proof}

Step $1$. $\overline{\mathbb{K}}^n = (\oplus_{j=1}^q Im \, B_j^*) \oplus Ker \, A$.

Indeed for every $v \in \overline{\mathbb{K}}^n$ we get that $v - \sum_{j=1}^q B_j^* B_j v$ is an element of $Ker \, A$ by standard applications of the properties of SVD systems and this allows to get that $\overline{\mathbb{K}}^n = (\sum_{j=1}^q Im \, B_j^*) + Ker \, A$.

If $\sum_{j=1}^q B_j^* v_j + w =0$ with $w \in Ker \, A$, then for every $h$ the SVD properties give: $B_h^* A (\sum_{j=1}^q B_j^* v_j + w) = \tau_h B_h^* v_h =0$ with $\tau_h > 0$, so $B_h^* v_h =0$ for every $h$ and so $w=0$ and the sum is direct.

Step $2$. We have $q \le p$, every $\tau_h$ is a singular value of $A$ and 

$Im \, B_h^* \subseteq Ker(A^*A - \tau_h^2 I_n)$.

For, it suffices to remark that, for every $h$, we have $B_h^* \neq 0$ and $A^* A B_h^* v = \tau_h^2 B_h^* v$ for every $v \in \overline{\mathbb{K}}^n$ and this follows again by the same properties.

Step $3$. We have: $q=p$ and, for every $h$, $\tau_h = \sigma_h$ and $Im \, B_h^* = Ker(A^*A - \tau_h^2 I_n)$.

By step $1$ we have $n = \sum_{i=1}^q dim (Im \, B_i^*) + dim(Ker \, A)$. By step $2$ and by $Ker(A^* A)= Ker(A)$ we have:
$n \le \sum_{i=1}^q dim(A^* A - \tau_i^2 I_n) + dim \, Ker(A^*A -0 I_n)$ 

$\le \sum_{i=1}^p dim(A^* A - \sigma_i^2 I_n) + dim \, Ker(A^*A -0 I_n) = n$, 
(the last equality because $A^* A$ is diagonalizable over $\overline{\mathbb{K}}$). 

Hence all inequalities are actually equalities and this is possible only if step $3$ holds.

Step $4$. We have $B_h = A_h$ for every $h$.

It suffices to check the equalities on the vectors of the same orthonormal basis of eigenvectors of $A^* A$ denoted by $\{ v_{ij} \ / \ 1 \le i \le p, \ 1 \le j \le n_i \} \cup \{w_1, \cdots , w_\nu \}$ in the proof of \ref{SVD-esiste}.
The properties of SVD systems give that, for every $h$, $B_h A^* A = \sigma_h^2 B_h$ and so $B_h = \dfrac{B_h A^* A}{\sigma_h^2}$. Hence $A_h w_j = 0 =B_h w_j$ for every $h,j$. 

By step $3$, for every $i,j$ there is $u_{ij} \in \overline{\mathbb{K}}^n$ such that $v_{ij} = B_i^* u_{ij}$. 

Now, for every $h,i,j$, $A_h v_{ij} = \dfrac{\delta_{hi}}{\sigma_h} A v_{hj} = \dfrac{\delta_{hi}}{\sigma_h} A B_h^* u_{hj}$ and this last is easily reduced to $B_h B_i^* u_{ij} = B_h v_{ij}$. This concludes the proof.
\end{proof}

\begin{remdef}
As shown above, every matrix $A \in  M_n(\overline{\mathbb{K}}) \setminus \{0\}$, where $\mathbb{K}$ is a real closed field, has a unique SVD: $A = \sum_{i=1}^p \sigma_i A_i$ and the values $\sigma_i$'s are the singular values of $A$.
We call the matrices $A_i$'s \emph{the SVD components of} $A$. 

By \ref{SVD-esiste}, the SVD components of $A$ have coefficients in $\mathbb{K}$ as soon as $A$ has coefficients in $\mathbb{K}$.
\end{remdef}

\end{document}